\documentclass[11pt,a4paper,regno]{article}

\usepackage[english]{babel}
\usepackage{graphicx} 
\usepackage{amsthm} 
\usepackage{amsmath}
\usepackage{amssymb}
\usepackage{amsthm} 
\usepackage{wasysym} 
\usepackage{mathrsfs} 
\usepackage{bbm} 
\usepackage{calligra} 
\usepackage{enumerate} 
\usepackage[shortlabels]{enumitem}
\usepackage{xcolor}
\usepackage{multirow} 
\usepackage{subfigure}
\usepackage{nicefrac}

\usepackage{hyperref}

\newcommand{\R}{\mathbb{R}}
\newcommand{\C}{\mathbb{C}}
\newcommand{\T}{\mathbb{T}}

\newcommand{\dif}{\,\mathrm{d}}

\def\XXint#1#2#3{{\setbox0=\hbox{$#1{#2#3}{\int}$}
		\vcenter{\hbox{$#2#3$}}\kern-.5\wd0}}

\theoremstyle{theorem}
\newtheorem{thm}{Theorem}[section]
\newtheorem{prop}{Proposition}[section]
\newtheorem{cor}{Corollary}[section]
\newtheorem{lemma}{Lemma}[section]
\theoremstyle{definition}
\newtheorem{definition}{Definition}[section]
\newtheorem{conjecture}{Conjecture}[section]
\theoremstyle{remark}

\title{Non-uniqueness of admissible solutions for the 2D Euler equation with $L^p$ vortex data}
\author{Francisco Mengual}
\date{\today}

\begin{document}
	
\maketitle

\begin{abstract}
For any $2<p<\infty$ we prove that there exists an initial velocity field $v^\circ\in L^2$ with vorticity $\omega^\circ\in L^1\cap L^p$ for which there are infinitely many bounded admissible solutions $v\in C_tL^2$ to the 2D Euler equation. This shows sharpness of the weak-strong uniqueness principle, as well as sharpness of Yudovich's proof of uniqueness in the class of bounded admissible solutions. The initial data are truncated power-law vortices. The construction is based on finding a suitable self-similar subsolution and then applying the convex integration method. In addition, we extend it for $1<p<\infty$ and show that the energy dissipation rate of the subsolution vanishes at $t=0$ if and only if $p\geq\nicefrac{3}{2}$, which is the Onsager critical exponent in terms of $L^p$ control on vorticity in 2D.
\end{abstract}

\section{Introduction and main results}

We consider the Cauchy problem for the Euler equation
\begin{subequations}\label{eq:IE}
\begin{align}
\partial_t v+\mathrm{div}(v\otimes v)+\nabla p &= 0, \\
\mathrm{div}v&=0,\\
v|_{t=0}&=v^\circ,
\end{align}
\end{subequations}
posed on the domain $[0,T]\times\R^2$, where $p(t,x)$ is the  pressure, $v(t,x)$ is the velocity field, and $v^\circ(x)$ is the initial datum. 
In this work we are interested in non-uniqueness of weak (i.e.~distributional) solutions to the Euler equation. More precisely, we address the question of what is the threshold regularity at $t=0$ for which uniqueness of bounded admissible solutions fails. 
A weak solution $v\in L_t^\infty L^2$ to the Euler equation is called \textit{admissible} if it does not increase the (kinematic) energy $E:=\frac{1}{2}\|v\|_{L^2}^2$
\begin{equation}\label{def:admissible}
E(t)\leq E(0)\quad\textrm{for a.e.}\quad t\in[0,T].
\end{equation}
This admissibility criterion is based on considering weakly convergent sequences of Leray solutions of Navier-Stokes with vanishing viscosity (see e.g.~\cite{DeLellisSzekelyhidi10}). Admissible solutions
coincide with strong solutions as long as the latter exist:
Suppose $v^\circ$ admits a strong solution $v_s\in C^1$, and let $v$ be another admissible solution. A straightforward computation shows that the relative energy $E_{\text{rel}}:=\frac{1}{2}\|v-v_s\|_{L^2}^2$ can be bounded by
\begin{equation}\label{eq:Erel:1}
E_{\text{rel}}(t)
\leq\int_0^t\int_{\R^2}|\nabla v_s||v-v_s|^2\dif x\dif\tau.
\end{equation}
This estimate combined with the Gr\"onwall inequality allows to conclude that necessarily $E_{\text{rel}}=0$ ($v=v_s$). 
Indeed, it is enough to assume that $\nabla v_s\in L_t^1 L^\infty$.
This fact is known in the literature as the weak-strong uniqueness principle (see e.g.~\cite{Wiedemann18}).

\begin{thm}[Weak-strong uniqueness principle]\label{thm:weakstrong} 
Suppose there exists a strong solution $v_s\in C^1$ to the Euler equation. Then, it is unique in the class of admissible solutions.
\end{thm}

Our first main result shows sharpness of Theorem \ref{thm:weakstrong} for H\"older spaces. More precisely, it states that if the $C^1$ assumption is weakened at a single point by $C^\gamma$ for some $0<\gamma<1$, then uniqueness fails in the class of admissible solutions. As a by-product, it shows existence of wild data above the Onsager critical exponent $\gamma=\nicefrac{1}{3}$ (see Section \ref{sec:CI}).

\begin{thm}\label{thm:nonuniqueness:Holder} For any $0<\gamma<1$ there exists a steady solution $v_s\in C^\gamma$ to the Euler equation which is smooth  away from the origin, with the property that there are infinitely many admissible solutions $v\in C_tL^2$ to the Euler equation \eqref{eq:IE} starting from $v^\circ=v_s$.
\end{thm}

Notice that Theorem \ref{thm:weakstrong} concerns uniqueness, while existence is just an assumption. Above $C^1$ regularity, Wolibner \cite{Wolibner33} and H\"older \cite{Holder33} proved global well-posedness of the 2D Euler equation in $C^{1,\gamma}$ for any $\gamma>0$ (assuming suitable decay as $|x|\to\infty$). 
In this class, uniqueness follows immediately from Theorem \ref{thm:weakstrong}. The proof of global existence exploits the fact that the vorticity $\omega=\mathrm{rot}v$ 
is transported by the flow
\begin{subequations}\label{eq:IEvorticity}
\begin{align}
\partial_t\omega+\mathrm{div}(v\omega)&= 0, \\
v&=\nabla^\perp\Delta^{-1}\omega,\label{eq:IEvorticity:2}\\
\omega|_{t=0}&=\omega^\circ,
\end{align}
\end{subequations}
where $\omega^\circ=\mathrm{rot}v^\circ$, and \eqref{eq:IEvorticity:2} is the Biot-Savart law 
\begin{equation}\label{BiotSavart}
v(x)^*=\frac{1}{2\pi i}\int_{\R^2}\frac{\omega(y)}{x-y}\dif y.
\end{equation}
In \eqref{BiotSavart} we identify $\R^2$ with the complex plane $\C$ as usual, where $i$ denotes the imaginary unit and $*$ the complex conjugate.
It is well known from Harmonic analysis that the map $\omega\mapsto \nabla^\perp\Delta^{-1}\omega=v$ is continuous from $C^\gamma$ to $C^{1,\gamma}$. For $C^1$ vector fields $v$, the trajectory map $X$ of the flow is well defined by the Cauchy-Lipschitz theory applied to
\begin{subequations}\label{eq:X}
\begin{align}
\partial_tX&=v(t,X),\\
X|_{t=0}&=\mathrm{id}.
\end{align}
\end{subequations}
Thus, the Euler equation \eqref{eq:IEvorticity} can be written as
$\omega(t,X)=\omega^\circ(x)$, where $\omega$ and $X$ are related implicitly through $v$. 
The rigorous proof of global existence in $C^{1,\gamma}$ is carried out by a Schauder fixed-point argument 
(see e.g.~\cite{Kato67}). This result is in stark contrast to the 3D case, where Elgindi \cite{Elgindi21} proved formation of finite-time singularities due to vortex stretching (local well-posedness was known since Lichtenstein \cite{Lichtenstein25} and Gunther \cite{Gunther27}).

The borderline case $C^1$ is more delicate.
In this class, Bourgain and Li \cite{BourgainLi15} and latter Elgindi and Masmoudi \cite{ElgindiMasmoudi20} 
proved strong ill-posedness for the Euler equation (see \cite{CMZO22} for strong ill-posedness in $H^\beta$). 
The reason behind is that the map $\omega\mapsto v$ sends bounded vorticities to  $\log$-Lipschitz velocities. In spite of the lack of $C^1$ regularity, Yudovich \cite{Yudovich63} showed that the  log-Lipschitz modulus of continuity is still valid to define uniquely the Lagrangian map $X$ and prove global well-posedness (see also \cite[Chapter 8]{MajdaBertozzi02}). This  fact makes the class of bounded vorticities a natural space for the 2D Euler equation.
In order to motivate our second main result it is convenient to recall Yudovich's proof of uniqueness. This can be understood as a refinement of the proof of the weak-strong uniqueness principle.

We start by recalling two classical estimates of the Biot-Savart operator \eqref{BiotSavart}. The first one is the boundedness of the map $\omega\mapsto v$ from $L^1\cap L^p$ to $L^\infty$ for any $2<p\leq\infty$. This follows by splitting $\R^2$ into $|x-y|\geq 1$ \& $|x-y|<1$, and then applying the H\"older inequality
\begin{equation}\label{eq:vinfty}
\|v\|_{L^\infty}
\leq \|\omega\|_{L^1}+\left(\frac{p-1}{p-2}\right)^{1-\nicefrac{1}{p}}\|\omega\|_{L^p}.
\end{equation}
The second estimate is the $L^p$-boundedness of the map $\omega\mapsto\nabla v$ for any $1<p<\infty$. 
Notice that \eqref{BiotSavart} is the Cauchy transform, and thus $\nabla v$ can be written in terms of the Beurling transform of $\omega$, a 2D version of the Hilbert transform.
Then, it follows from the Calderon-Zygmund theory that
\begin{equation}\label{eq:Dvp}
\|\nabla v\|_{L^p}
\leq C\frac{p^2}{p-1}\|\omega\|_{L^p}.
\end{equation}
We will use $C$ to denote a constant, which may change from line to line but will be universal. The bound \eqref{eq:Dvp}  gives the exact growth as $p\to 1,\infty$. The precise computation of the $L^p$-norm of the Beurling transform is an outstanding open problem related to the Morrey conjecture (see e.g.~\cite{AIPS12}). Here we just need the inequality $\|\nabla v\|_{L^p}\leq Cp\|\omega\|_{L^p}$ for $2<p<\infty$.

Next, we recall Yudovich's energy method. 
Let $v_s\in C_t L^2$ with $\omega_s\in L_t^\infty(L^1\cap L^\infty)$ be a (Yudovich) solution, and let $v$ be another bounded admissible solution with $v^\circ=v_s^\circ$. By applying the H\"older inequality and \eqref{eq:Dvp}, the r.h.s.~of \eqref{eq:Erel:1} can be bounded by
\begin{equation}\label{eq:Erel:2}
\int_{\R^2}|\nabla v_s||v-v_s|^2\dif x
\leq Cp\|\omega_s\|_{L^p}\|v-v_s\|_{L^\infty}^{\nicefrac{2}{p}} E_{\text{rel}}^{1-\nicefrac{1}{p}},
\end{equation}
for any $2<p<\infty$, 
which plugged into \eqref{eq:Erel:1} implies that
\begin{equation}\label{eq:Erel:3}
E_{\text{rel}}(t)\leq \|v-v_s\|_{L_{t,x}^\infty}^2(C\|\omega_s\|_{L^p}t)^p.
\end{equation}
We recall that $\|\omega_s\|_{L^p}$ is independent of time because $\omega_s(t,X)=\omega_s^\circ(x)$ with $X$ volume-preserving by $\mathrm{div}v_s=0$.
On the one hand, the term $\|v-v_s\|_{L_{t,x}^\infty}$ can be bounded by the $L_{t,x}^\infty$-norm of $v$ and $v_s$ separately: 
the first is bounded by hypothesis, and the latter by \eqref{eq:vinfty} for $p=\infty$.
On the other hand, by the log-convexity of the $L^p$-norms we have
\begin{equation}\label{eq:omegasp}
\|\omega_s\|_{L^p}
\leq\max\{\|\omega_s\|_{L^1},\|\omega_s\|_{L^\infty}\}<\infty.
\end{equation} 
Finally, by letting $p\to\infty$ in \eqref{eq:Erel:3}, it follows that necessarily $E_{\text{rel}}=0$ ($v=v_s$).

\begin{thm}[Yudovich's well-posedness Theorem]\label{thm:Yudovich} Let $v^\circ\in L^2$ with $\omega^\circ\in L^1\cap L^\infty$ and $\mathrm{div}v^\circ=0$. Then, there exists a global solution $v\in C_t L^2$ with $\omega=L_t^\infty(L^1\cap L^\infty)$ to the Euler equation. Furthermore, it is unique in the class of bounded admissible solutions.
\end{thm}

Our second main result shows sharpness of Yudovich's proof of uniqueness for $L^p$ spaces. More precisely, it states that if the $L^\infty$ assumption is weakened at a single point by $L^p$ for some $2<p<\infty$, then uniqueness fails in the class of bounded admissible solutions.

\begin{thm}\label{thm:nonuniqueness:Lp} For any $2< p<\infty$ there exists a steady solution $v_s\in L^2$ with $\omega_s\in L^1\cap L^p$ to the Euler equation which is smooth  away from the origin, with the property that there are infinitely many bounded admissible solutions $v\in C_tL^2$ to the Euler equation \eqref{eq:IE} starting from $v^\circ=v_s$.
\end{thm}

Let us point out where Yudovich's proof of uniqueness is not working in Theorem \ref{thm:nonuniqueness:Lp}. Notice that the term $\|v-v_s\|_{L_{t,x}^\infty}$ can still be bounded by the $L_{t,x}^\infty$-norm of $v$ and $v_s$ separately: 
the first is bounded by hypothesis, and the latter by \eqref{eq:vinfty} for $2<p<\infty$. However, now the condition \eqref{eq:omegasp} fails, namely we have $\|\omega_s\|_{L^q}=\infty$ for $q>p$, which prevents from concluding $E_{\mathrm{rel}}=0$ via \eqref{eq:Erel:3}.  Remarkably, Yudovich \cite{Yudovich95} extended his uniqueness result for unbounded vorticities for which $\|\omega_s\|_{L^p}$ has moderate growth as $p\to\infty$. Let us recall Taniuchi's (non-localized) version \cite{Taniuchi04} of this generalization (see also \cite{CMZ19}): Given a non-decreasing function $\Theta:[1,\infty)\to [1,\infty)$, a vorticity $\omega$ belongs to the \textit{Yudovich space} $Y^\Theta$ if
$$\|\omega\|_{Y^\Theta}:=\sup_{p\in[1,\infty)}\frac{\|\omega\|_{L^p}}{\Theta(p)}<\infty.$$
Then, the Euler equation \eqref{eq:IEvorticity} is globally well-posed in $Y^\Theta$ if $\Theta$ satisfies the Osgood type condition
\begin{equation}\label{eq:Osgood}
\int_2^\infty\frac{\dif p}{p\Theta(p)}=\infty.
\end{equation}
Notice that Theorem  \ref{thm:Yudovich} corresponds to the particular case $Y^1=L^1\cap L^\infty$. In this regard, it would be interesting to explore if Theorem \ref{thm:nonuniqueness:Lp} could be extended to Yudovich spaces for which \eqref{eq:Osgood} fails.\\ 

We finish the introduction with several remarks on 
Theorems \ref{thm:nonuniqueness:Holder} \& \ref{thm:nonuniqueness:Lp}.
\begin{enumerate}[(i)]
	\item Our solutions have compact support. Hence, the same results hold in any arbitrary open subset of $\R^2$, as well as in the periodic domain $\T^2$.
	\item\label{rem:data} Theorem \ref{thm:nonuniqueness:Holder} follows from Theorem \ref{thm:nonuniqueness:Lp} by the Sobolev embedding $W^{1,p}\subset C^{1-\nicefrac{2}{p}}$ for $2<p<\infty$. Furthermore, both are corollaries of our third main result Theorem \ref{thm:energy}. We have chosen to introduce them separately for clarity of presentation. In fact, we take the same $v_s$ in all the theorems. They are truncated power-law vortices
	\begin{equation}\label{eq:vs}
	v_s(x)=\chi(|x|)|x|^{-\alpha}x^\perp,
	\end{equation}
	where $\chi$ is a smooth cutoff function, and $0<\alpha< 1$ is a parameter. With this choice we have
	$v_s\in C^{1-\alpha}$ and $\omega_s\in L^{\nicefrac{2}{\alpha}-}$. The time of existence depends on the truncation, and it can be made arbitrarily large (see \eqref{eq:T}).
	\item\label{rem:construction} Our solutions $v\in C_tL^2$ are obtained by means of convex integration. They equal $v_s$ outside a disc of radius $(ct)^{\nicefrac{1}{\alpha}}$, where $0<c\leq C\alpha$ are constants that will be specified in Section \ref{sec:subsolution}. As a result, $v|_{t=0}=v_s$ and they are smooth outside $\{|x|\leq (ct)^{\nicefrac{1}{\alpha}}\}$. Inside this region
	we only know that the vorticity $\omega$ is a distribution. The question of  non-uniqueness of vorticities in 
	$L_t^\infty(L^1\cap L^p)$ remains open (see Section \ref{sec:SS}). In spite of the lack of uniqueness and regularity, these velocities are close in average to a subsolution of the form
	\begin{equation}\label{eq:barv}
	\bar{v}(t,x)=\chi(|x|)\frac{h(t,|x|)}{|x|}x^\perp,
	\end{equation}
	where $h$ is a self-similar profile
	\begin{equation}\label{eq:h}
	h(t,r)=(ct)^{\frac{1-\alpha}{\alpha}}H(\xi),
	\quad\quad
	\xi=\frac{r}{(ct)^{\nicefrac{1}{\alpha}}}.
	\end{equation}
	We will declare $H(\xi)=\xi^{1-\alpha}$ for $\xi\geq 1$, or equivalently $\bar{v}=v_s$ for $|x|\geq (ct)^{\nicefrac{1}{\alpha}}$. Our central task will be therefore to find a suitable profile $H$ on $[0,1]$.
	\item The construction explained in \ref{rem:construction} is also valid for $1\leq\alpha <2$. Thus, Theorem \ref{thm:nonuniqueness:Lp} holds for $1<p\leq 2$, but removing the property ``bounded''. The borderline case $\alpha\to2$ corresponds to a point vortex. This will be analyzed in Section \ref{sec:alpha2}.
	\item We can impose our solutions to conserve the energy. The energy dissipation rate of the subsolution vanishes at $t=0$ if and only if $\alpha <\nicefrac{4}{3}$. This corresponds to the Onsager critical exponent $p=\nicefrac{3}{2}$ in terms of $L^p$ control on vorticity in 2D (see Section \ref{sec:energy}).
\end{enumerate}

\begin{thm}\label{thm:energy} 
	There exists a subsolution $(\bar{v},\bar{\sigma},\bar{q})$ to the Euler equation which agrees with \eqref{eq:vs} outside $\{|x|\leq (ct)^{\nicefrac{1}{\alpha}}\}$.
	Furthermore, the energy dissipation rate equals
	$$\partial_t\bar{E}
	=-\frac{\pi}{16}
	c^{\frac{8-\alpha}{2\alpha}} t^{\frac{4-3\alpha}{\alpha}}
	\quad
	\textrm{with}
	\quad
	c=\left(\frac{2\alpha}{4-\alpha}\right)^2.
	$$
	In particular, $\partial_t\bar{E}|_{t=0}=0$ if and only if $\alpha <\nicefrac{4}{3}$, or equivalently $\omega^\circ\in L^{\nicefrac{3}{2}}$.
\end{thm}

\begin{figure}[h!]
	\centering
	\includegraphics[width=0.71\textwidth]{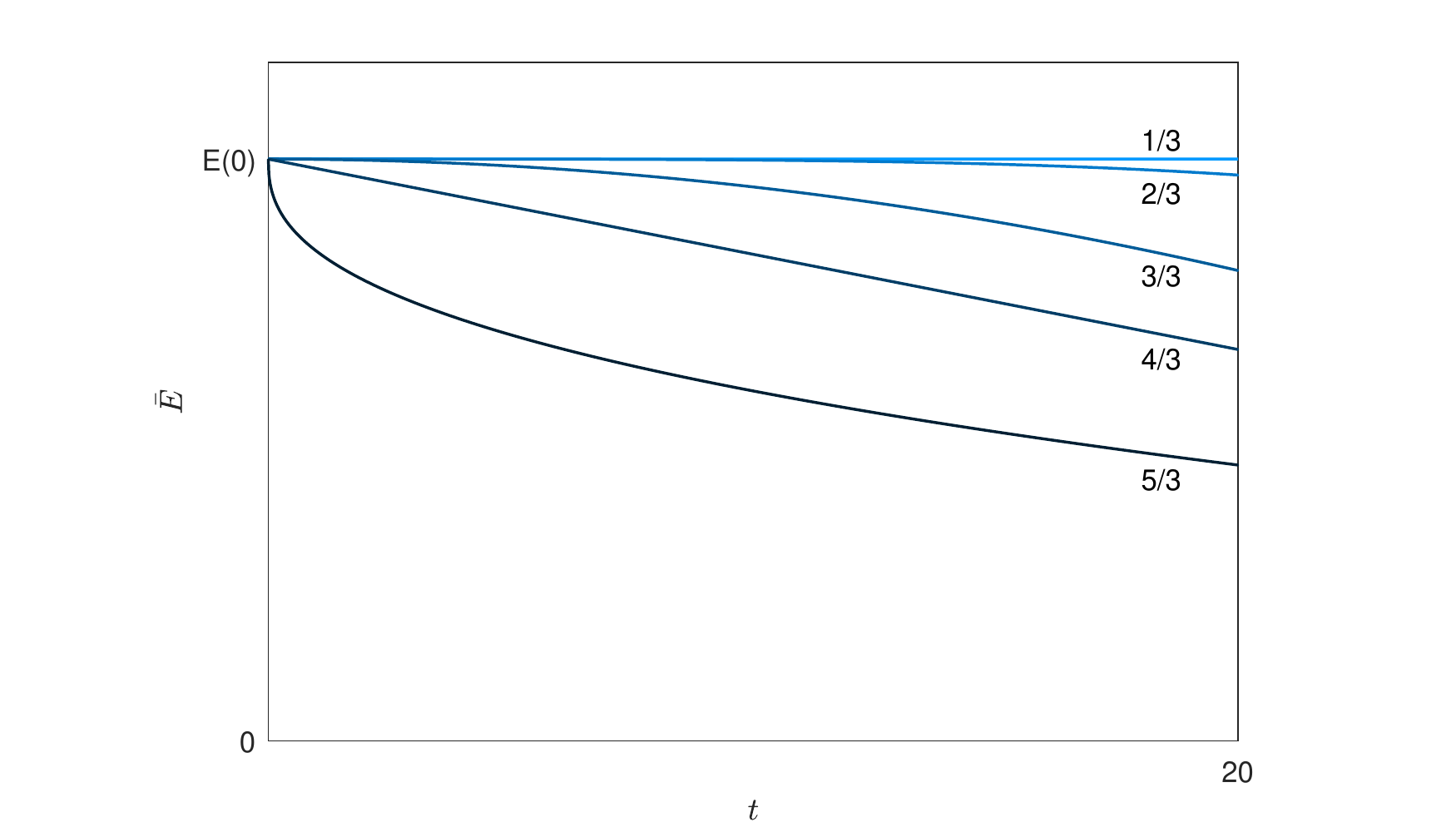}
	\caption{From lighter to darker blue,
		plot of the energy $\bar{E}(t)$ of the subsolution $(\bar{v},\bar{\sigma},\bar{q})$ for the powers $\alpha=\nicefrac{k}{3}$ with $k=1,2,3,4,5$. The initial energy $E(0)$ is taken independently of $\alpha$. For $\alpha=\nicefrac{1}{3},\nicefrac{2}{3}$ the energy is almost constant. The power $\alpha=1$ corresponds to the threshold for bounded velocities.
		The power $\alpha=\nicefrac{4}{3}$ corresponds both to the threshold for $L^{\nicefrac{3}{2}}$ vorticities and $\partial_t\bar{E}|_{t=0}=0$. For $\alpha=\nicefrac{5}{3}$ the energy decreases faster.}
	\label{fig:energy}
\end{figure}

\begin{figure}[h!]
	\centering
	\includegraphics[width=0.71\textwidth]{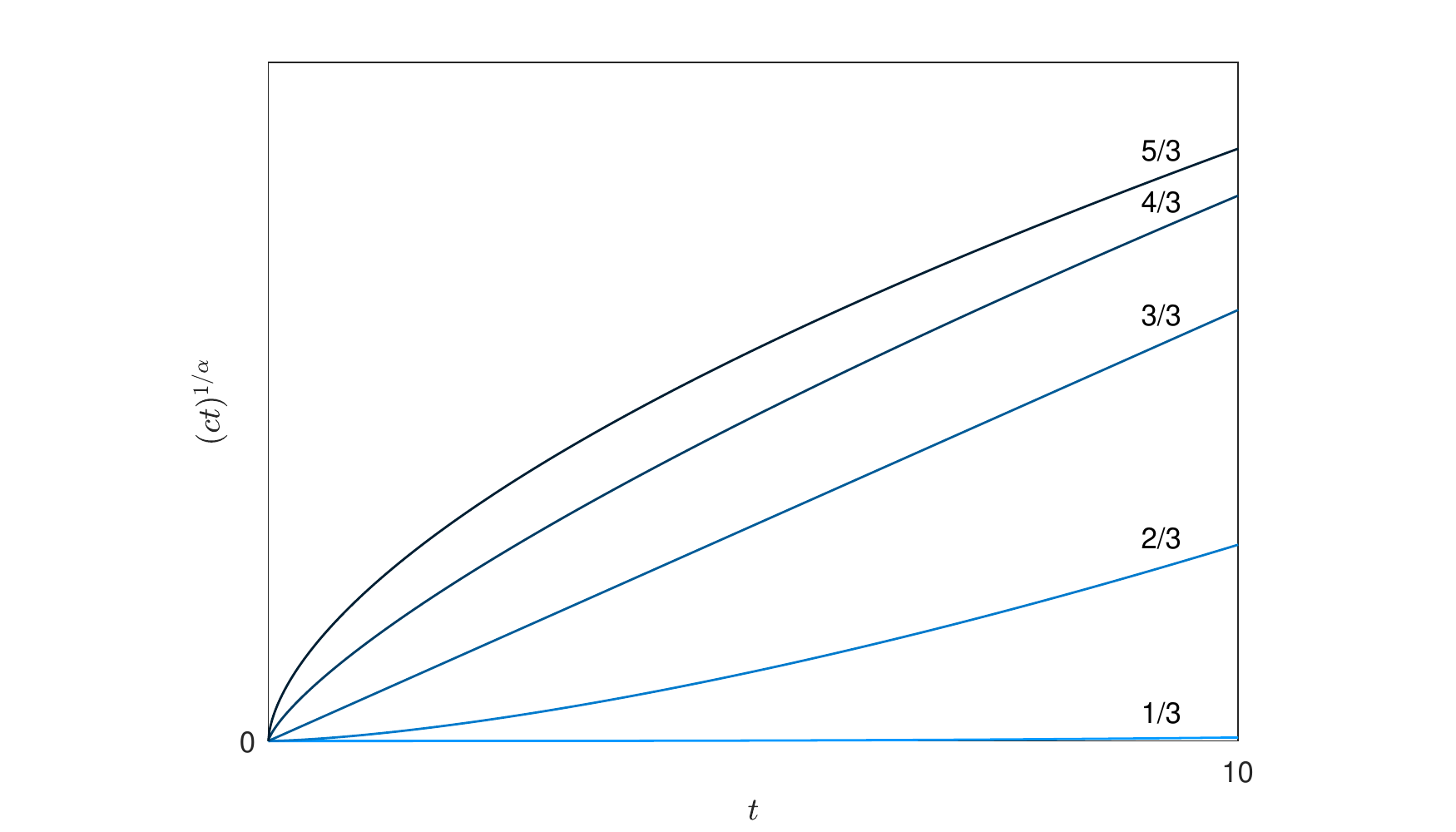}
	\caption{From lighter to darker blue,
		plot of the boundary $(ct)^{\nicefrac{1}{\alpha}}$ for the powers $\alpha=\nicefrac{k}{3}$ with $k=1,2,3,4,5$. This shows how the region $\{|x|\leq (ct)^{\nicefrac{1}{\alpha}}\}$ shrinks as $\alpha\to 0$.}
	\label{fig:turzone}
\end{figure}

\section{Brief background}\label{sec:background}

In this section we review briefly the literature on non-uniqueness, energy conservation and admissibility criteria for the Euler equation and compare it with the present work.

\subsection{Self-similarity and symmetry breakdown}\label{sec:SS}
In the context of Yudovich's well-posedness Theorem, if the initial vorticity is not bounded but at least $\omega^\circ\in L^1\cap L^p$ for some $1<p<\infty$, DiPerna and Majda \cite{DiPernaMajda87} proved the existence of a global solution $v\in L_t^\infty L^2$ with $\omega\in L_t^\infty(L^1\cap L^p)$ to the Euler equation (see also \cite[Chapter 10]{MajdaBertozzi02}). 
However, uniqueness is not expected to hold in general in this class. Roughly speaking, for $2<p<\infty$ the Sobolev embedding implies continuity of the velocity field, and thus existence of the Lagrangian map $X$ by the Peano Theorem, but the modulus of continuity does not satisfy the Osgood uniqueness criterion. Let us formulate this question as a conjecture, which remains open to the best of our knowledge. 

\begin{conjecture}\label{conjecture:Lp} For any $2<p<\infty$ there exists $v^\circ\in L^2$ with $\omega^\circ\in L^1\cap L^p$, with the property that there is more than one weak solution $v\in C_t L^2$ with $\omega\in L_t^\infty(L^1\cap L^p)$ to the Euler equation \eqref{eq:IEvorticity}.
\end{conjecture}

Global existence for the 2D Euler equation \eqref{eq:IEvorticity} has been proved in other vorticity classes (see e.g.~the
recent work of Crippa and Stefani \cite{CrippaStefani21} and the references therein).
In \cite{Taniuchi04} 
Taniuchi proved global existence in the Yudovich space $Y^\Theta$ if the Osgood condition \eqref{eq:Osgood} is weakened by
\begin{equation}\label{eq:Osgoodlog}
\int_2^\infty\frac{\dif p}{p\Theta(\log p)}=\infty.
\end{equation}
More precisely, he proved global existence in uniformly-localized Yudovich spaces including $\mathbf{bmo}$. In this regard, Conjecture \ref{conjecture:Lp} could be also stated for Yudovich spaces satisfying the existence condition \eqref{eq:Osgoodlog} but not the uniqueness condition \eqref{eq:Osgood}.

In the recent groundbreaking work \cite{Vishik18I,Vishik18II} Vishik solved Conjecture \ref{conjecture:Lp} for the forced  Euler equation (see also the notes \cite{ABCDGJK21})
\begin{subequations}\label{eq:IEvorticityforced}
\begin{align}
\partial_t\omega+\mathrm{div}(v\omega)&=f, \label{eq:IEvorticityforced:1}\\
v&=\nabla^\perp\Delta^{-1}\omega,\\
\omega|_{t=0}&=\omega^\circ.
\end{align}
\end{subequations}

\begin{thm}[Vishik's non-uniqueness Theorem]\label{thm:Vishik} For any $2<p<\infty$ there exists $v^\circ\in L^2$ with $\omega^\circ\in L^1\cap L^p$ and a force $f=\mathrm{rot}g\in L_t^1(L^1\cap L^p)$ with $g\in L_t^1 L^2$, with the property that there are infinitely many weak solutions $v\in C_t L^2$ with $\omega\in L_t^\infty(L^1\cap L^p)$ to the forced  Euler equation \eqref{eq:IEvorticityforced}.
\end{thm}

Theorem \ref{thm:nonuniqueness:Lp} has certain connections with Vishik's non-uniqueness Theorem. Firstly, both theorems share the same initial data \eqref{eq:vs}. Secondly, Vishik's construction is also based on finding a suitable self-similar velocity $\bar{v}$. Furthermore, $\bar{v}$ is a modification of the power-law vortex $\beta|x|^{-\alpha}x^\perp$ in a disc of radius $2t^{\nicefrac{1}{\alpha}}$, where $\beta$ is a sufficiently large constant. Thirdly, $\bar{v}$ is also truncated by $\chi$ to guarantee integrability at infinity. In spite of these similarities, both the results and the proofs differ significantly. 
Concerning the results, Vishik's non-uniqueness Theorem solves Conjecture \ref{conjecture:Lp} in the natural vorticity class $L^1\cap L^p$ by introducing a force, while Theorem \ref{thm:nonuniqueness:Lp} shows non-uniqueness without forcing by considering distributional vorticities. Our original motivation was indeed to explore the possibility of removing the force in Theorem \ref{thm:Vishik} by means of convex integration, but paying the price of the low-regularity inherent to these constructions. The first obvious attempt was to absorb Vishik's force into the Reynolds stress. However, it was not immediate for us that the corresponding subsolution was admissible. More precisely, we first needed to derive the conditions under which a radially symmetric self-similar subsolution yields admissible solutions via convex integration. After this, instead of checking if Vishik's vortex satisfies these conditions, it becomes easier to construct our own profile $H$ (recall \eqref{eq:h}). This is because Vishik's condition for $H$ is much less explicit, namely it must have an unstable eigenvalue $\lambda$ associated to the Rayleigh stability equation. 
In a \textit{tour de force}, Vishik proved in \cite{Vishik18I,Vishik18II} the existence of such an unstable vortex $\bar{v}$ and, after defining \textit{ad hoc} the force $f$ by \eqref{eq:IEvorticityforced:1}, he checked carefully the existence of other $m$-fold symmetric velocities $v$ deviating from $\bar{v}$ as $t^\lambda$ (see also \cite{ABCDGJK21}).
Remarkably, Albritton, Bru\'e, and Colombo \cite{ABC22} proved recently non-uniqueness of Leray solutions for the forced 3D Navier-Stokes equation by adapting properly Vishik's unstable vortex into the cross section of an axisymmetric without swirl vortex ring. This approach is framed within the program of Jia, \v{S}ver\'ak, and Guillod \cite{JiaSverak15,GuillodSverak17} on the conjectural non-uniqueness of Leray solutions of the (unforced) 3D Navier-Stokes equation.

In the recent investigation \cite{BressanMurray20,BressanShen21} Bressan, Murray, and Shen showed numerical evidence toward the validity of Conjecture \ref{conjecture:Lp}. Their work is also based on self-similarity and symmetry breakdown: the initial vorticity is of the form $\omega^\circ=\omega_s\Omega$, where $\omega_s$ is the vorticity of \eqref{eq:vs} and $\Omega$ is a suitable smooth function which only depends on the polar angle. 
Their approach suggests two different ways of regularizing $\omega^\circ$ leading to either one or two algebraic spirals.
In contrast to Vishik's spectral analysis, their construction relies on a smart system of adapted coordinates due to Elling \cite{Elling16a,Elling16b} (see also the recent work of Garc\'ia and G\'omez-Serrano for the generalized SQG equation \cite{GarciaGomezSerrano22}).

\subsection{Convex integration}\label{sec:CI}
In this work we deal with a weaker version of Conjecture \ref{conjecture:Lp}: the integrability condition $\omega\in L_t^\infty(L^1\cap L^p)$ is removed (for $t>0$) and then necessarily \eqref{eq:IEvorticity} is replaced by \eqref{eq:IE}. The first result in this direction is due to Scheffer \cite{Scheffer93}: there exist Euler velocities $v\in L_{t,x}^2$ with compact support in space-time. Latter, this construction was simplified by Shnirelman \cite{Shnirelman97}. In the seminal work \cite{DeLellisSzekelyhidi09} De Lellis and Sz\'ekelyhidi proved the same result in the energy space $L_t^\infty L^2$, and for any space dimension $d\geq 2$, by adapting Gromov's convex integration method and Tartar's plane wave analysis into Hydrodynamics. Observe that these solutions show non-uniqueness for the trivial initial datum $v^\circ=0$. Non-uniqueness in $L_t^\infty L^2$ was generalized by Wiedemann \cite{Wiedemann11} for every divergence-free $v^\circ\in L^2$ (see \cite{KMY22} for recent improvements of the regularity). For smooth initial data, the aforementioned solutions necessarily increase the energy, as a consequence of the weak-strong uniqueness principle. In this sense, a divergence-free $v^\circ\in L^2$ is called \textit{wild} if it admits infinitely many admissible solutions $v\in L_t^\infty L^2$.
In \cite{DeLellisSzekelyhidi10} De Lellis and Sz\'ekelyhidi initiated the investigation on non-uniqueness of admissible solutions, upon which this work is based. Recall that the property ``admissible'' includes both conservative ($E=E(0)$) and dissipative ($E<E(0)$) solutions. 
In his famous work \cite{Onsager49} Onsager conjectured, in the context of the zeroth law of turbulence, the threshold regularity for the validity of the energy conservation of weak solutions to the Euler equation (in $\T^3$). Onsager's conjecture, which is nowadays a theorem, can be stated as follows:
\begin{enumerate}[(a)]
	\item\label{Onsager:a} Any weak solution $v\in C_t C^\gamma$ to the Euler equation with $\gamma>\nicefrac{1}{3}$ conserves the energy.
	\item\label{Onsager:b} For any $0<\gamma<\nicefrac{1}{3}$ there exist weak solutions $v\in C_t C^\gamma$ to the Euler equation which do not conserve the energy.
\end{enumerate}
Part \ref{Onsager:a} was fully proved by Constantin, E, and Titi \cite{CET94}, after a partial result of Eyink \cite{Eyink94}. Part \ref{Onsager:b} was solved more recently by Isett \cite{Isett18}, and by Buckmaster, De Lellis, Sz\'ekelyhidi, and Vicol \cite{BDSV19}. 
The last achievement took a decade of refinements of the convex integration method, and the study of its connection with turbulent flows is still an active research area (see e.g.~the recent work of Novack and Vicol \cite{NovackVicol23}). 
Coming back to the initial value problem, Theorem \ref{thm:nonuniqueness:Holder} is presumably the first example of wild data with H\"older regularity above the Onsager critical exponent $\nicefrac{1}{3}\leq\gamma<1$ (see \cite{BKP23} for a convex integration construction in $C^{\nicefrac{1}{2}-}$ for the forced 3D Euler equation). Below the Onsager critical exponent $0<\gamma<\nicefrac{1}{3}$, Daneri and Sz\'ekelyhidi \cite{DaneriSzekelyhidi21}
proved that the set of wild data $v^\circ\in C^{\gamma}$ is a dense subset of the
divergence-free vector fields in $L^2$ (see also \cite{DeRosaTione22}). However, to the best of our knowledge, it was not known neither concrete examples of wild data in $C^\gamma$, nor even the size of their set of singular points (see \cite{DeRosaHaffter22} for estimates of the singular set of times). 
In this regard, Theorem \ref{thm:nonuniqueness:Holder} provides an explicit example with a minimal singular set: a single point $\{0\}$. We remark that, although the aforementioned literature on the Onsager conjecture \ref{Onsager:b} is posed on the 3D periodic domain $\T^3$, our solutions can be trivially adjust to this setting. However, in this case the singular set becomes a line $\{0\}\times\T$. On this matter, it would be interesting to explore if there might be other 3D wild data with smaller singular sets.

\subsection{Energy conservation/dissipation}\label{sec:energy}

In this section we discuss the energy conservation/dissipation of weak solutions to the Euler equation in terms of the vorticity. For simplicity of presentation we consider the periodic domain. The same results hold in the euclidean space \textit{mutatis mutandis}. 

We start by recalling that Onsager's conjecture \ref{Onsager:a} is a corollary of the following Besov type criterion (see e.g.~\cite{DuchonRobert00,CCFS08}): Any weak solution $v$ to the Euler equation in $\T^d$ satisfying
\begin{equation}\label{eq:OnsagerBesov}
\lim_{|y|\to 0}\int_0^T\int_{\T^d}\frac{|v(t,x+y)-v(t,x)|^3}{|y|}\dif x\dif t
= 0,
\end{equation}
conserves the energy. Although this condition is independent of the dimension $d>1$, it has stronger implications in the 2D case. As a first easy consequence, by the Sobolev embedding $W^{1,p}\subset C^{1-\nicefrac{2}{p}}$ and \ref{Onsager:a}, it follows that the energy is conserved for any weak solution with $L^p$ control on vorticity for $p>3$. However, the energy conservation can be extended for smaller $p$'s by taking more advantage of \eqref{eq:OnsagerBesov}. 
Let us recall the argument from \cite[Proposition 6]{DuchonRobert00}. Firstly, by applying the H\"older inequality to $\delta_y v=v(x+y)-v(x)$,
$$\|\delta_y v\|_{L^3}
\leq\|\delta_y v\|_{L^p}^\theta\|\delta_y v\|_{L^q}^{1-\theta}
\quad\textrm{with}\quad
\frac{1}{3}=\frac{\theta}{p}+\frac{1-\theta}{q}.$$
Secondly, by using the Sobolev embedding $W^{1,p}\subset L^q$ for $q=\frac{2p}{2-p}$ and 
$\|\delta_y v\|_{L^p}\leq|y|\|v\|_{W^{1,p}}$, 
$$\|\delta_y v\|_{L^3}\leq C|y|^\theta\|v\|_{W^{1,p}}
\quad\textrm{with}\quad
\theta=\frac{5}{3}-\frac{2}{p}.$$
Therefore, the energy conservation criterion \eqref{eq:OnsagerBesov} is satisfied for $\theta>\nicefrac{1}{3}$, or equivalently $p>\nicefrac{3}{2}$.
This result was extended to the borderline case $p=\nicefrac{3}{2}$ by
Cheskidov, Lopes Filho, Nussenzveig Lopes, and Shvydkoy \cite{CLNS16}. Furthermore, they constructed a velocity field with vorticity in $L^{\nicefrac{3}{2}-}$ exhibiting non-vanishing energy flux. This suggests that there might exist dissipative solutions of the 2D Euler equation with vorticity in $L^{\nicefrac{3}{2}-}$. This threshold is somehow related to our construction: the energy dissipation rate of our subsolution vanishes at $t=0$ if and only if  $\alpha <\nicefrac{4}{3}$, where recall $\omega^\circ=\omega_s\in L^{\nicefrac{2}{\alpha}-}$. 
Let us give an heuristic explanation of this phenomenon.

First of all, it is convenient to rewrite the Euler equation \eqref{eq:IE} as 
\begin{subequations}\label{eq:IEq}
\begin{align}
\partial_t v+\mathrm{div}(v\ocircle v)+\nabla q &= 0, \\
\mathrm{div}v&=0,\\
v|_{t=0}&=v^\circ,
\end{align}
\end{subequations}
where
$$v\ocircle v:=v\otimes v-\frac{1}{2}|v|^2I_2,$$
is the projection of $v\otimes v$ into the space of
symmetric and traceless matrices, and the trace has been absorbed by the \textit{Bernoulli pressure}
$$q:=p+\frac{1}{2}|v|^2.$$

Let us analyze the non-uniqueness/energy-dissipation scenario in the Euler equation. Firstly, we introduce some notation. Concerning non-uniqueness: Given $v^\circ\in L^2$ with $\mathrm{div}v^\circ=0$ and $\mathrm{rot}v^\circ\in L^1\cap L^p$, let us denote $S=S(v^\circ)$ by the space of admissible solutions to the Euler equation, and $S_p=S_p(v^\circ)$ by the subset of $S$ formed by velocities with $\omega\in L_t^\infty(L^1\cap L^p)$. Recall that $S_p$ is non-empty by \cite{DiPernaMajda87}. 
With this notation, we can rewrite Yudovich's well-posedness Theorem: if $p=\infty$, then $S=S_\infty=\{v\}$ where $v$ is the Yudovich solution. Similarly, we can reformulate Conjecture \ref{conjecture:Lp}: for any $2<p<\infty$ there exists $v^\circ$ satisfying $|S_p(v^\circ)|>1$, and Theorem \ref{thm:nonuniqueness:Lp}: $|S(v^\circ)|=\infty$ for the initial data \eqref{eq:vs}.
Concerning energy dissipation: Given 
$v\in S$ we denote its energy by
$$E:=\int_{\T^2} e\dif x,
\quad\quad
e:=\frac{1}{2}|v|^2,$$
and its dissipation by
$$D:=E(0)-E.$$

Next, we give a condition for non-uniqueness/energy-dissipation in terms of averaged solutions.
Given $\mu$ a probability measure on $S$, we define 
$$\bar{v}:=\int_{S}v\dif\mu,
\quad\quad
\bar{\sigma}
:=\int_{S}(v\ocircle v)\dif\mu,
\quad\quad
\bar{q}
:=\int_{S}q\dif\mu.$$
The triple $(\bar{v},\bar{\sigma},\bar{q})$ satisfies the relaxed Euler equation
\begin{subequations}\label{eq:IER}
\begin{align}
\partial_t\bar{v}+\mathrm{div}\bar{\sigma}+\nabla\bar{q} &= 0, \\
\mathrm{div}\bar{v}&=0,\\
\bar{v}|_{t=0}&=v^\circ.
\end{align}
\end{subequations}
The (relaxed) energy was obtained by De Lellis and Sz\'ekelyhidi in \cite[Lemma 3]{DeLellisSzekelyhidi10}
\begin{equation}\label{eq:sub:e}
\bar{E}=\int_{\T^2}\bar{e}\dif x,
\quad\quad
\bar{e}
:=\frac{1}{2}|\bar{v}|^2 +\lambda_{\max}(\bar{v}\ocircle\bar{v}-\bar{\sigma}),
\end{equation}
where $\lambda_{\max}$ denotes the largest eigenvalue.
The crucial facts are that $\bar{e}$ is convex and agrees with $e$ on $S$.  
By applying the Jensen inequality and the Fubini-Tonelli Theorem, we get
\begin{subequations}\label{eq:barE}
\begin{align}
\bar{E}
&=\int_{\T^2}\bar{e}(\bar{v},\bar{\sigma})\dif x
=\int_{\T^2}\bar{e}\left(\int_{S}(v,v\ocircle v)\dif\mu\right)\mathrm{d}x\\
&\leq
\int_{\T^2}\int_{S}\bar{e}(v,v\ocircle v)\dif\mu\dif x
=\int_S\int_{\T^2}e(v)\dif x\dif\mu\label{eq:barE:Jensen}\\
&=E(0)-\int_S D(v)\dif\mu.\label{eq:barE:D}\\ \nonumber
\end{align}
\end{subequations}
In particular, $\bar{E}\leq E(0)$. The last inequality is strict in two situations: 
\begin{enumerate}[(1)]
	\item\label{Jensen} If the Jensen inequality \eqref{eq:barE:Jensen} is strict. In this case, $\bar{e}$ is strictly convex on (the convex envelope) of $\mathrm{spt}(\mu)$. In particular, $|\mathrm{spt}(\mu)|>1$.
	\item\label{D} If the dissipative term in \eqref{eq:barE:D} is non-vanishing. In this case, $D(v)>0$ $\mu$-a.e.~$v\in S$. 
\end{enumerate}
In short, \ref{Jensen} concerns non-uniqueness and \ref{D} concerns energy dissipation. 
On the one hand, \ref{Jensen} is to be expected in the regime $1<p<\infty$. On the other hand, \ref{D} cannot occur if $\mathrm{spt}(\mu)\subset S_p$ in the regime $\nicefrac{3}{2}\leq p<\infty$, while it might be possible if $1<p<\nicefrac{3}{2}$. 

We have seen that the existence of a solution $(\bar{v},\bar{\sigma},\bar{q})$ to \eqref{eq:IER} with strictly decreasing energy $\bar{E}$ is linked to either non-uniqueness or energy dissipation. Our third main result Theorem \ref{thm:energy} shows the existence of this solution $(\bar{v},\bar{\sigma},\bar{q})$, which is called a subsolution in the convex integration framework, for the initial data \eqref{eq:vs}. 
Remarkably, the initial energy dissipation of the subsolution is imperceptible ($\partial_t\bar{E}|_{t=0}=0$) in the regime $\nicefrac{3}{2}\leq p<\infty$. Below $p=\nicefrac{3}{2}$, the energy dissipation rate is more abrupt, which might reflect the contribution of the dissipation \ref{D} to the Jensen gap \ref{Jensen}.

By virtue of the h-principle for the Euler equation discovered by De Lellis and Sz\'ekelyhidi \cite{DeLellisSzekelyhidi10} Theorem \ref{thm:energy} implies the existence of infinitely many admissible solutions, and thus Theorems \ref{thm:nonuniqueness:Holder} \& \ref{thm:nonuniqueness:Lp} follow as corollaries. Furthermore, the flexibility of the h-principle allows to prescribe any energy functional $e$ strictly greater than $\bar{e}$ on the region $\Omega=\{|x|\leq (ct)^{\nicefrac{1}{\alpha}}\}$. Thus, it is possible to select $e$ making the energy $E=\int e\dif x$ either constant (conservative) or decreasing (dissipative).
Here we do not consider increasing energies (non-admissible). 

\subsection{Other admissibility criteria}
In spite of these non-uniqueness results, there exist other admissibility criteria that allow to rule out some of these solutions. The first criterion is to belong to the natural space $S_p$.
In the regime $\nicefrac{3}{2}\leq p<\infty$, any weak solution $v\in S$ with decreasing energy cannot belong to $S_p$ because $S_p\subset S_{\text{con}}:=S\cap\{E=E(0)\}$. 
In contrast, by taking $E$ decreasing in the regime $1<p<\nicefrac{3}{2}$, Theorem \ref{thm:nonuniqueness:Lp} shows presumably the first example of dissipative solutions $v\in S$ with $L^p$ vorticity data, although this does not necessarily imply that $S_p\nsubseteq S_{\text{con}}$.
The second criterion is the vanishing viscosity limit. It was shown in \cite[Theorem 2]{CLNS16} that any physically realizable
solution conserves the energy for every $1<p\leq\infty$. This scaling gap in the energy conservation between ideal solutions and ideal limits has been
observed in other equations of Hydrodynamics (see \cite{CIN18} for SQG and \cite{FaracoLindberg20} for MHD). 
By taking $E$ constant in the regime $1<p<\infty$, Theorem \ref{thm:nonuniqueness:Lp} shows that $|S_{\text{con}}|=\infty$ for the initial data \eqref{eq:vs}, although this does not necessarily implies that $|S_p|>1$. For radially symmetric solutions, a third way to rule out non-uniqueness could be the stability of 2D viscous vortices (see e.g.~\cite{GallayWayne05}). In such a case, the proof of non-uniqueness of $L^p$ physical solutions would require less symmetric initial data. Another prerequisite satisfied by physicial solutions is the local energy (in)equality (see e.g.~\cite{DuchonRobert00,DeLellisSzekelyhidi10}). Globally dissipative solutions have been constructed in the last years via convex integration (see e.g.~\cite{Isett22,DeLellisHyunju22,GebhardKolumban22b}). We do not explore this property here.

We remark in passing that 2D anomalous dissipation of energy is to be expected in the borderline case $p\to 1$ (see e.g.~\cite{TBM20}). Remarkably, Sz\'ekelyhidi \cite{Szekelyhidi11} provided indeed the first concrete example of a wild datum in this class: the unstable vortex sheet $v^\circ(x)=\mathrm{sgn}(x_2)$. 
This initiated a promising program on modeling hydrodynamical instabilities via convex integration: see \cite{Szekelyhidi11,MengualSzekelyhidi23} for the Kelvin-Helmholtz, \cite{GKS21,GebhardKolumban22a,GHK22} for the Rayleigh-Taylor, and \cite{CFG11,Szekelyhidi12,CCF21,ForsterSzekelyhidi18,CFM19,HitruhinLindberg21,Mengual22,NoisetteSzekelyhidi21,CFM22} for the Saffman-Taylor instabilities.

\subsection{Organization of the paper}

We start Section \ref{sec:Hprinciple} by writing the relaxed Euler equation \eqref{eq:IER} for radially symmetric self-similar subsolutions. In Section \ref{sec:subsolution} we derive the conditions under which an admissible subsolution exists, and we construct an example. Then, we prove Theorems \ref{thm:nonuniqueness:Holder}, \ref{thm:nonuniqueness:Lp} \& \ref{thm:energy}. Finally, we analyze the borderline case $\alpha\to 2$ in Section \ref{sec:alpha2}.

\section{H-principle for symmetric subsolutions}\label{sec:Hprinciple}

In this section we write the h-principle for the Euler equation of De Lellis and Sz\'ekelyhidi \cite{DeLellisSzekelyhidi10} for 2D radially symmetric self-similar subsolutions. We start by recalling the definitions of weak solution and subsolution to the Euler equation. 

\begin{definition}
A triple $(\bar{v},\bar{\sigma},\bar{q})\in C_t(L^2\times L^1\times L^1)$ where
\begin{equation}\label{eq:sub:1}
\bar{v}=\left[\begin{array}{c}
\bar{v}_1 \\ \bar{v}_2
\end{array}\right],
\quad\quad
\bar{\sigma}
=\left[\begin{array}{rr}
\bar{\sigma}_1 & \bar{\sigma}_2 \\
\bar{\sigma}_2 & -\bar{\sigma}_1
\end{array}\right],\quad\quad
\bar{q},
\end{equation}
is a \textit{subsolution} to the Euler equation \eqref{eq:IEq} if $\bar{v}$ is weakly divergence-free and
\begin{equation}\label{eq:weak}
\int_0^T\int_{\R^2}(\bar{v}\cdot\partial_t\Phi+\bar{\sigma}:\nabla\Phi + \bar{q}\mathrm{div}\Phi)\dif x\dif t=-\int_{\R^2}v^\circ\cdot\Phi|_{t=0}\dif x,
\end{equation}
holds for every test function $\Phi\in C_c^1([0,T)\times\R^2)$. The pair $(\bar{v},\bar{q})$ is a \textit{weak solution} to the Euler equation if $(\bar{v},\bar{v}\ocircle\bar{v},\bar{q})$ is a subsolution.
\end{definition}

Next, we recall the h-principle for the Euler equation  \cite[Proposition 2]{DeLellisSzekelyhidi10}. Recall the definition of the energy functional $\bar{e}$ \eqref{eq:sub:e}.

\begin{thm}[H-principle for the Euler equation]\label{thm:HP} Let $\Omega$ be a non-empty open subset of $(0,T]\times\R^2$ and let $e\in C(\Omega)$ with $e\mathbf{1}_{\Omega}\in C_tL^1$. Suppose there exists a subsolution $(\bar{v},\bar{\sigma},\bar{q})$ to the Euler equation satisfying the following properties:
\begin{itemize}
	\item $\{\bar{\sigma}\neq\bar{v}\ocircle\bar{v}\}
	\subseteq\Omega$.
	\item $(\bar{v},\bar{\sigma})$ maps continuously $\Omega$ into $\{\bar{e}<e\}$.
\end{itemize}
Then, there exist infinitely many weak solutions $(v,q)$ to the Euler equation with Bernoulli's pressure $q=\bar{q}$ and velocity $v$ satisfying
\begin{align*}
v=\bar{v}&\quad\textrm{outside }\Omega,\\
\frac{1}{2}|v|^2=e
&\quad\textrm{inside }\Omega.
\end{align*}
\end{thm}

By virtue of Theorem \ref{thm:HP}, the proof of non-uniqueness of admissible solutions to the Euler equation is reduced to find a subsolution $(\bar{v},\bar{\sigma},\bar{q})$ with non-empty $\Omega$ and strictly decreasing energy $\bar{E}=\int\bar{e}\dif x$. In this case, we will say that $(\bar{v},\bar{\sigma},\bar{q})$ is an \textit{admissible subsolution}.

\subsection{Complex coordinates}

In this section we rewrite the relaxed Euler equation \eqref{eq:IER} in complex coordinates
$$x=x_1+ix_2.$$
In this setting, for any $z,w\in\R^2\simeq\C$ we denote as usual
$$|z|=\sqrt{z_1^2+z_2^2},
\quad\quad
z^*=z_1-iz_2,
\quad\quad
z^\perp=iz=-z_2+iz_1,
$$
and also
\begin{align*}
z\cdot w&=(zw^*)_1=z_1w_1+z_2w_2,\\
z\cdot w^\perp&=(zw^*)_2=z_2w_1-z_1w_2.
\end{align*}
By slightly abuse of the notation, we identify
$$\nabla=\partial_1+i\partial_2,$$
and $(\bar{v},\bar{\sigma},\bar{q})$ in \eqref{eq:sub:1} with
$$
\bar{v}=\bar{v}_1+i\bar{v}_2,
\quad\quad
\bar{\sigma}=\bar{\sigma}_1 + i\bar{\sigma}_2,
\quad\quad
\bar{q}=\bar{q}+i0.
$$

\begin{prop}\label{prop:complex} The relaxed Euler equation \eqref{eq:IER} is written in complex coordinates as
\begin{subequations}\label{eq:IERcomplex}
	\begin{align}
	\partial_t\bar{v}+\nabla^*\bar{\sigma}+\nabla\bar{q} &= 0, \\
	\nabla\cdot\bar{v}&=0,\\
	\bar{v}|_{t=0}&=v^\circ,
	\end{align}
\end{subequations}
and the energy \eqref{eq:sub:e} equals
\begin{equation}\label{eq:sub:e:complex}
\bar{e}=\frac{1}{2}|\bar{v}|^2+\left|\frac{1}{2}\bar{v}^2-\bar{\sigma}\right|.
\end{equation}
Furthermore, the subsolution is a solution to the Euler equation if and only if
$$\bar{\sigma}=\frac{1}{2}\bar{v}^2.$$
\end{prop}
\begin{proof}
On the one hand,
\begin{align*}
\mathrm{div}
\left[\begin{array}{rr}
\bar{\sigma}_1 & \bar{\sigma}_2 \\
\bar{\sigma}_2 & -\bar{\sigma}_1
\end{array}\right]
&=\left[\begin{array}{c}
\partial_1\bar{\sigma}_1 + \partial_2\bar{\sigma}_2 \\
\partial_1\bar{\sigma}_2 - \partial_2\bar{\sigma}_1
\end{array}\right]
=
(\partial_1-i\partial_2)(\bar{\sigma}_1+i\bar{\sigma}_2).
\end{align*}
On the other hand,
$$\bar{v}\ocircle\bar{v}
=\frac{1}{2}\left[\begin{array}{cc}
\bar{v}_1^2-\bar{v}_2^2 & 2\bar{v}_1\bar{v}_2 \\[0.1cm]
2\bar{v}_1\bar{v}_2 & \bar{v}_2^2-\bar{v}_1^2
\end{array}\right]
=
\frac{1}{2}(\bar{v}_1+i\bar{v}_2)^2.$$
For \eqref{eq:sub:e:complex} it is easy to check that any traceless symmetric matrix $z$ satisfies
$$\lambda_{\max}
\left[\begin{array}{rr}
z_1 & z_2 \\
z_2 & -z_1
\end{array}\right]=\sqrt{z_1^2+z_2^2}=|z|,$$
where we identify $z=z_1+iz_2$. Finally, if $(\bar{v},\bar{q})$ is a solution to the Euler equation, it holds
$$\nabla^*\left(\frac{1}{2}\bar{v}^2-\bar{\sigma}\right)=0.$$
Then, since $\frac{1}{2}\bar{v}^2-\bar{\sigma}$ is anti-holomorphic and integrable, necessarily $\frac{1}{2}\bar{v}^2-\bar{\sigma}=0$.
\end{proof}

\subsection{Radial symmetry}
In this section we write the relaxed Euler equation \eqref{eq:IERcomplex} in polar coordinates
$$x=re^{i\theta},$$
for radially symmetric subsolutions. More precisely, we assume that the fluid is rotating around the origin: the velocity $\bar{v}$ is of the form
\begin{equation}\label{ansatz:v}
\bar{v}(t,x):=h(t,r)ie^{i\theta},
\end{equation}
for some real-valued $h$, to be determined. 
Under this choice, $\bar{v}$ is automatically divergence-free (see \eqref{eq:D*v}) and the vorticity $\bar{\omega}=\mathrm{rot}\bar{v}$ is radially symmetric $\bar{\omega}(t,x)=g(t,r)$, where $h$ and $g$ are related by
$$rg=\partial_r(rh).$$
In order to compare $\bar{\sigma}$ with $\frac{1}{2}\bar{v}^2=-\frac{1}{2}h^2e^{2i\theta}$, it seems convenient to take $\bar{\sigma}$ of the form
\begin{equation}\label{ansatz:s}
\bar{\sigma}(t,x):=-w(t,r)e^{2i\theta},
\end{equation}
for some complex-valued $w$, to be determined.
Finally, we also assume (although it can be deduced from the equation) that the Bernoulli pressure $\bar{q}$ is radially symmetric, and then (by slightly abuse of the notation) we write
\begin{equation}\label{ansatz:q}
\bar{q}(t,x):=q(t,r),
\end{equation}
for some (real-valued) $q$, to be determined.

\begin{prop}\label{prop:polar} Under the choice \eqref{ansatz:v}-\eqref{ansatz:q}, the relaxed Euler equation \eqref{eq:IERcomplex} is written as
\begin{subequations}\label{eq:IERpolar}
\begin{align}
i\partial_t h-\frac{\partial_r(r^2w)}{r^2}+\partial_r q &= 0,\label{eq:IERpolar:1} \\ 
h|_{t=0}&=h^\circ,
\end{align}
\end{subequations}
where $v^\circ(x)=h^\circ(r)ie^{i\theta}$, 
and the energy \eqref{eq:sub:e:complex} equals
\begin{equation}\label{eq:sub:e:polar}
\bar{e}=\frac{1}{2}h^2+\left|\frac{1}{2}h^2-w\right|.
\end{equation}
Furthermore, the subsolution is a solution to the Euler equation if and only if
$$w=\frac{1}{2}h^2.$$
In this case, $h$ is steady and $q$ satisfies
$$\partial_rq=gh.$$
\end{prop}
\begin{proof}
By writing the gradient in polar coordinates
$$\nabla_x=e^{i\theta}\left(\partial_r+\frac{i}{r}\partial_\theta\right),$$
we deduce that
\begin{align*}
\nabla_x^*\bar{v}
&=e^{-i\theta}
\left(\partial_r-\frac{i}{r}\partial_\theta\right)(hie^{i\theta})
=i\frac{\partial_r(rh)}{r},\\
\nabla_x^*\bar{\sigma}
&=e^{-i\theta}
\left(\partial_r-\frac{i}{r}\partial_\theta\right)(-we^{2i\theta})
=-e^{i\theta}\frac{\partial_r(r^2w)}{r^2},
\end{align*}
and also
$$\nabla_x\bar{q}
=e^{i\theta}\partial_r q.$$
On the one hand (recall $h$ is real-valued)
\begin{subequations}\label{eq:D*v}
\begin{align}
\nabla_x\cdot\bar{v}
&=(\nabla_x^*\bar{v})_1=0,\\
\bar{\omega}=\nabla_x^\perp\cdot\bar{v}
&=(\nabla_x^*\bar{v})_2=\frac{\partial_r(rh)}{r}
=g.
\end{align}
\end{subequations}
On the other hand,
$$\partial_t\bar{v}
+\nabla_x^*\bar{\sigma} + \nabla_x\bar{q}
=e^{i\theta}\left(i\partial_t h
-\frac{\partial_r(r^2w)}{r^2}+\partial_r q\right).$$
We have proved \eqref{eq:IERpolar}. 
The equality \eqref{eq:sub:e:polar} follows from the definitions \eqref{ansatz:v}\eqref{ansatz:s}.
Finally, by decomposing \eqref{eq:IERpolar:1} into its real and imaginary part respectively, we deduce that $w=w_1+iw_2$ and $(q,h)$ are related by
\begin{subequations}\label{eq:IERpolarparts}
\begin{align}
r^2\partial_r q&=\partial_r(r^2w_1),\\
r^2\partial_t h&=\partial_r(r^2w_2).
\end{align}
\end{subequations}
Hence, if $w=\frac{1}{2}h^2$ we have $\partial_th=0$, and $\partial_r q=gh$ follows from
\begin{equation}\label{eq:qgh}
\frac{1}{2}\partial_r(rh)^2
=r^2gh.
\end{equation}
This concludes the proof.
\end{proof}

\begin{cor}\label{cor:radial} The energy \eqref{eq:sub:e:polar} is minimized in $w_1$ by taking
\begin{equation}\label{ansatz:w1}
w_1:=\frac{1}{2}h^2.
\end{equation}
Under the choice \eqref{ansatz:w1}, the relaxed Euler equation \eqref{eq:IERpolar} is written as
\begin{subequations}\label{eq:IERpolar:h}
\begin{align}
\partial_r q&=gh,\\
\partial_r(r^2w_2)&=r^2\partial_ th,
\\ 
h|_{t=0}&=h^\circ,
\end{align}
\end{subequations}
and the energy \eqref{eq:sub:e:polar} equals
\begin{equation}\label{eq:sub:e:radial}
\bar{e}=\frac{1}{2}h^2+|w_2|.
\end{equation}
Furthermore, the subsolution is a solution to the Euler equation if and only if
$$w_2=0.$$
In this case, $h$ is steady.
\end{cor}
\begin{proof}
It follows from \eqref{eq:sub:e:polar}, \eqref{eq:IERpolarparts} and \eqref{eq:qgh}.
\end{proof}

By virtue of Corollary \ref{cor:radial}, the functions $w,q$ and $\bar{e}$ are determined by $h$. Therefore, the construction of an admissible subsolution is reduced to find a profile $h$ with non-vanishing $w_2$ and strictly decreasing energy $\bar{E}=\int\bar{e}\dif x$.

\subsection{Scaling symmetry}

In this section we write the relaxed Euler equation \eqref{eq:IERpolar:h} for self-similar subsolutions. 
The (relaxed) Euler equation possess a two-parameter scaling symmetry (see e.g.~\cite{ABCDGJK21}): 
If $(\bar{v},\bar{\sigma},\bar{q})$ is a (sub)solution and $\lambda,\mu>0$, then
\begin{equation}\label{ansatz:SS}
\bar{v}_{\lambda,\mu}(t,x)
=\frac{\lambda}{\mu}\bar{v}(\lambda t,\mu x),
\quad
\bar{\sigma}_{\lambda,\mu}(t,x)
=\Big(\frac{\lambda}{\mu}\Big)^2\bar{\sigma}(\lambda t,\mu x),
\quad
\bar{q}_{\lambda,\mu}(t,x)
=\Big(\frac{\lambda}{\mu}\Big)^2\bar{q}(\lambda t,\mu x),
\end{equation}
define another (sub)solution. This corresponds to the physical dimensions
$$[x]=L,
\quad\quad
[t]=T,
\quad\quad
[\bar{v}]=\frac{L}{T}
,
\quad\quad
[\bar{\sigma}]=[\bar{q}]=\Big(\frac{L}{T}\Big)^2.$$
We say that $(\bar{v},\bar{\sigma},\bar{q})$ is \textit{self-similar} if it is invariant under the scaling $L^\alpha\sim T$ for some $\alpha>0$, that is, if $(\bar{v},\bar{\sigma},\bar{q})_{\lambda,\mu}=(\bar{v},\bar{\sigma},\bar{q})$ for all $\lambda,\mu>0$ given by the relation
$$\lambda=\nicefrac{1}{t}=c\mu^{\alpha},$$
for some parameters $\alpha,c>0$. 

We assume that the triple $(\bar{v},\bar{\sigma},\bar{q})$ given by \eqref{ansatz:v}-\eqref{ansatz:q} and \eqref{ansatz:w1} is self-similar for some $\alpha,c>0$. Then, the profiles $(h,w_2,q)$ are of the form
\begin{equation}\label{ansatz:h}
h(t,r):=(ct)^{\frac{1-\alpha}{\alpha}}H(\xi),
\quad
w_2(t,r):=-\frac{c}{\alpha}(ct)^{\frac{2(1-\alpha)}{\alpha}}W_2(\xi),
\quad
q(t,r):=
(ct)^{\frac{2(1-\alpha)}{\alpha}}Q(\xi),
\end{equation}
in self-similar variables
$$\xi:=\frac{r}{(ct)^{\nicefrac{1}{\alpha}}},$$
for some functions $(H,W_2,Q)$, to be determined. Under this choice, the vorticity profile is of the form (see \eqref{eq:G})
$$g(t,r)
:=\frac{1}{ct}G(\xi),$$
where $H$ and $G$ are related by
$$\xi G=\partial_\xi(\xi H).$$

\begin{prop}\label{prop:h2}
Under the choice \eqref{ansatz:h}, 
 the relaxed Euler equation \eqref{eq:IERpolar:h} is written as
\begin{subequations}\label{eq:IER:SS}
\begin{align}
\partial_\xi Q
&=GH,\label{eq:IER:SS:1}\\
\partial_\xi(\xi^2W_2)
&=\xi^{4-\alpha}\partial_\xi(\xi^{\alpha-1}H),\label{eq:IER:SS:2}\\
\lim_{\xi\to\infty}\xi^{\alpha-1}H(\xi)
&=\beta,\label{eq:IER:SS:3}
\end{align}
where $h^\circ(r)=\beta r^{1-\alpha}$ for some $\beta\in\R$, and the energy \eqref{eq:sub:e:radial} equals
$$\bar{e}
=(ct)^{\frac{2(1-\alpha)}{\alpha}}
\left(\frac{1}{2}H^2+\frac{c}{\alpha}|W_2|\right).$$
Furthermore, the subsolution is a solution to the Euler equation if and only if
$$W_2=0.$$
In this case, $H(\xi)=\beta\xi^{1-\alpha}$.
\end{subequations}
\end{prop}
\begin{proof}
First of all, we compute
\begin{equation}\label{eq:ht}
\partial_t h
=-\frac{c}{\alpha}(ct)^{\frac{1-2\alpha}{\alpha}}\xi^{2-\alpha}\partial_\xi(\xi^{\alpha-1}H),
\quad\quad
\partial_r h
=\frac{1}{ct}\partial_\xi H.
\end{equation}
On the one hand, since
\begin{equation}\label{eq:G}
g=\frac{1}{r}\partial_r(rh)
=\frac{1}{ct}\frac{1}{\xi}\partial_\xi(\xi H)
=\frac{1}{ct}G,
\end{equation}
the equation \eqref{eq:IER:SS:1} follows from
$$
(ct)^{\frac{1-2\alpha}{\alpha}}\partial_\xi Q
=
\partial_r q
=gh
=
(ct)^{\frac{1-2\alpha}{\alpha}}GH.
$$
On the other hand, the equation \eqref{eq:IER:SS:2} follows from
$$
-\frac{c}{\alpha}(ct)^{\frac{3-2\alpha}{\alpha}}\partial_\xi(\xi^2 W_2)
=\partial_r(r^2 w_2)
=r^2\partial_th
=-\frac{c}{\alpha}(ct)^{\frac{3-2\alpha}{\alpha}}\xi^{4-\alpha}\partial_\xi(\xi^{\alpha-1}H).$$
The equation \eqref{eq:IER:SS:3} follows from
$$h^\circ(r)
=\lim_{t\to 0}h(t,r)
=r^{1-\alpha}\lim_{\xi\to\infty}\xi^{\alpha-1}H(\xi).$$
The rest follows from \eqref{ansatz:h} and \eqref{eq:IER:SS:2}.
\end{proof}

By virtue of Proposition \ref{prop:h2}, the functions $W_2,Q$ and $\bar{e}$ are determined by $H$, which is now time-independent (in contrast to $h$). 
Notice that the condition \eqref{eq:IER:SS:3} prevents from constructing subsolutions with finite energy. However, this inconvenient can be easily fixed by truncating the profile $H$ (see Section \ref{sec:truncation}). 
Therefore, the construction of an admissible subsolution is reduced to find a profile $H$ with non-vanishing $W_2$ and satisfying $\int\partial_t\bar{e}\dif x<0$.

\section{Admissible subsolutions}\label{sec:subsolution}

In this section we prove Theorems \ref{thm:nonuniqueness:Holder}, \ref{thm:nonuniqueness:Lp} \& \ref{thm:energy} by constructing first admissible subsolutions, and then invoking the h-principle for the 2D Euler equation.

The first step is to construct radially symmetric self-similar subsolutions $(\bar{v},\bar{\sigma},\bar{q})$. These are given by the choices \eqref{ansatz:v}-\eqref{ansatz:q}, \eqref{ansatz:w1} and \eqref{ansatz:h} in terms of the parameters $\alpha,c>0$, and some functions $(H,W_2,Q)$ which must satisfy the equation \eqref{eq:IER:SS}.

As it is stated in Proposition \ref{prop:h2}, this subsolution $(\bar{v},\bar{\sigma},\bar{q})$ is a solution to the Euler equation if and only if $W_2=0$, and so $H(\xi)=\beta\xi^{1-\alpha}$. In contrast to Vishik's work \cite{Vishik18I,Vishik18II}, here the constant $\beta$ does not play a crucial role, and thus we will take $\beta=1$ for simplicity.
The profile $H(\xi)=\xi^{1-\alpha}$ corresponds to the steady power-law vortex $\bar{v}=v_s$
$$v_s(x)=|x|^{-\alpha}x^\perp.$$
In this case, the velocity and vorticity profiles are given by
\begin{align*}
h_s(r)=r^{1-\alpha},\
\quad\quad
g_s(r)=(2-\alpha)r^{-\alpha},
\end{align*}
and the Bernoulli pressure equals
\begin{equation}\label{eq:qs}
q_s(r)
=\left\lbrace
\begin{array}{cl}
\frac{2-\alpha}{2(1-\alpha)}r^{2(1-\alpha)}, & \alpha\neq 1,\\[0.1cm]
\ln r, & \alpha=1.
\end{array}
\right.
\end{equation}
Notice that $(v_s,q_s)\in L_{\text{loc}}^2\times L_{\text{loc}}^1$ if and only if $\alpha<2$, and also $\omega_s\in L_{\text{loc}}^p$ for $p<\nicefrac{2}{\alpha}$.

As we mentioned at the end of Section \ref{sec:Hprinciple}, we need to find a profile $H$ with non-vanishing $W_2$ and satisfying $\int\partial_t\bar{e}\dif x<0$. Since we want to minimize the contribution of $W_2$ to the energy, we assume that the subsolution agrees with the power-law vortex outside $[0,1]$
\begin{equation}\label{ansatz:Hpower}
H(\xi):=\xi^{1-\alpha},\quad\quad\xi>1.
\end{equation}
We also impose the regularity conditions: $H(0)=0$, $H(1)=1$, and $H\in C^1([0,1])$. 
The condition $H(0)=0$ is necessary to make $\bar{v}$ continuous at $x=0$ for $t>0$.
Hence, it remains to determine $H$ in the interval $(0,1)$.
Next, we need to guarantee that the support of $W_2$ is indeed contained in $[0,1]$. This yields the first condition for $H$.

\begin{lemma}[1st condition for $H$] Under the choice \eqref{ansatz:Hpower}, the support of the solution $W_2$ to \eqref{eq:IER:SS:2} is contained in $[0,1]$ if and only if $H$ satisfies
\begin{equation}\label{eq:H:condition:1}
(4-\alpha)\int_0^1\xi^2H\dif\xi=1.
\end{equation}
\end{lemma}
\begin{proof}
The solution $W_2$ to \eqref{eq:IER:SS:2} is given by
$$W_2(\xi)=\frac{1}{\xi^2}\left(\int_0^\xi
\zeta^{4-\alpha}\partial_\zeta(\zeta^{\alpha-1}H)\dif\zeta
+C\right),$$
for some constant $C$.
Since the profile $\nicefrac{1}{\xi^2}$ is neither continuous nor locally integrable at $\xi=0$, necessarily $C=0$. By \eqref{ansatz:Hpower}, we have $W_2(\xi)=0$ for $\xi\geq 1$ if and only if $W_2(1)=0$. Finally, an integration by parts yields
\begin{equation}\label{ew:W2:1}
W_2(\xi)=\xi H-\frac{4-\alpha}{\xi^2}\int_0^\xi\zeta^2 H\dif\zeta,
\end{equation}
from which we deduce \eqref{eq:H:condition:1}.
\end{proof}

\begin{prop}\label{prop:QW2}
Suppose \eqref{ansatz:Hpower}\eqref{eq:H:condition:1} hold. Then, the solution to the relaxed Euler equation \eqref{eq:IER:SS} is given by
\begin{equation}\label{eq:Q}
Q(\xi)
:=\left\lbrace
\begin{array}{cl}
\displaystyle q_s(1)-\int_{\xi}^{1}GH\dif\zeta, & 0<\xi\leq 1,\\[0.4cm]
q_s(\xi), & \xi>1,
\end{array}
\right.
\end{equation}
with $q_s$ as in \eqref{eq:qs}, and
\begin{equation}\label{eq:W2}
W_2(\xi)
:=\left\lbrace
\begin{array}{cl}
\displaystyle \xi H-\frac{4-\alpha}{\xi^2}\int_{0}^{\xi}\zeta^2H\dif\zeta, & 0<\xi\leq 1,\\[0.4cm]
0, & \xi>1.
\end{array}
\right.
\end{equation}
As a result, $(\bar{v},\bar{\sigma},\bar{q})=(v_s,v_s\ocircle v_s,q_s)$ outside $\{|x|\leq (ct)^{\nicefrac{1}{\alpha}}\}$.
\end{prop}
\begin{proof} The initial condition \eqref{eq:IER:SS:3} is automatically satisfied by \eqref{ansatz:Hpower}. The other two equations \eqref{eq:IER:SS:1}\eqref{eq:IER:SS:2} can be integrated (recall \eqref{ew:W2:1}).
\end{proof}

\subsection{The energy}

In this section we compute the energy dissipation rate $\int\partial_t\bar{e}\dif x$, and derive the conditions under which it becomes negative. Recall that the energy $\bar{e}$ is determined by $H$ through the choices 
\eqref{ansatz:v}-\eqref{ansatz:q}, \eqref{ansatz:w1}, \eqref{ansatz:h},  \eqref{ansatz:Hpower} and \eqref{eq:H:condition:1}.

\begin{prop} It holds
\begin{equation}\label{eq:energydissipation:1}
\int_{\R^2}\partial_t\bar{e}\dif x
=-\frac{2\pi}{\alpha} c^{\frac{2(2-\alpha)}{\alpha}}
\left(A-\frac{2(2-\alpha)}{\alpha}Bc\right)
t^{\frac{4-3\alpha}{\alpha}},
\end{equation}
where
\begin{subequations}
\begin{align}
A&:=\frac{1}{2}-(2-\alpha)\int_0^1\xi H^2\dif\xi,\label{def:A}\\
B&:=\int_0^1\left|\xi^2 H-\frac{4-\alpha}{\xi}\int_{0}^{\xi}\zeta^2H\dif\zeta\right|\dif\xi.\label{def:B}
\end{align}
\end{subequations}
\end{prop}
\begin{proof}
First of all, by recalling \eqref{eq:sub:e:radial} we write
$$\int_{\R^2}\partial_t\bar{e}\dif x
=\int_{\R^2}\partial_t\left(\frac{1}{2}h^2+|w_2|\right)\mathrm{d}x.$$
On the one hand, by applying \eqref{ansatz:h}, \eqref{eq:ht} and  \eqref{ansatz:Hpower}, we compute
$$
\frac{1}{2}\int_{\R^2}\partial_th^2\dif x
=2\pi\int_0^{(ct)^{\nicefrac{1}{\alpha}}} h\partial_th r\dif r
=-\frac{2\pi c}{\alpha}(ct)^{\frac{4-3\alpha}{\alpha}}
A,
$$
where
$$
A
=\int_0^1\xi^{3-\alpha}H\partial_\xi(\xi^{\alpha-1}H)\dif\xi
=\frac{1}{2}\int_0^1\xi^{2(2-\alpha)}\partial_\xi(\xi^{\alpha-1}H)^2\dif\xi
=\frac{1}{2}-(2-\alpha)\int_0^1\xi H^2\dif\xi.
$$
On the other hand, by applying \eqref{ansatz:h} and \eqref{eq:W2}, we compute
$$\int_{\R^2}|w_2|\dif x
=2\pi\int_0^{(ct)^{\nicefrac{1}{\alpha}}}|w_2|r\dif r
=\frac{2\pi c}{\alpha}(ct)^{\frac{2(2-\alpha)}{\alpha}}B,$$
where
$$B=\int_0^1|W_2|\xi\dif\xi
=\int_0^1\left|\xi^2 H-\frac{4-\alpha}{\xi}\int_{0}^{\xi}\zeta^2H\dif\zeta\right|\dif\xi.$$
Hence,
$$\int_{\R^2}\partial_t|w_2|\dif x
=\partial_t
\int_{\R^2}|w_2|\dif x
=\frac{2\pi c}{\alpha}\frac{2(2-\alpha)}{\alpha}(ct)^{\frac{4-3\alpha}{\alpha}}Bc.$$
This concludes the proof
\end{proof} 

\begin{cor}[2nd condition for $H$] Suppose $\int\partial_t\bar{e}\dif x<0$. Then, necessarily $A>0$, or equivalently
\begin{equation}\label{eq:H:condition:2}
2(2-\alpha)\int_0^1\xi H^2\dif\xi <1.
\end{equation}
\end{cor}

\subsection{The growth rate $c$}
In this section we select $c$ maximizing the energy dissipation rate.
\begin{prop} 
Suppose \eqref{eq:H:condition:2} holds. Then, $\int\partial_t\bar{e}\dif x<0$ if and only if
\begin{equation}\label{eq:c}
0<c<\frac{\alpha}{2(2-\alpha)}\frac{A}{B}.
\end{equation}
Furthermore, the energy dissipation rate is maximized
\begin{equation}\label{eq:energydissipation:2}
\int_{\R^2}\partial_t\bar{e}\dif x
=-2\pi
\left(\frac{A}{4-\alpha}\right)^{\frac{4-\alpha}{\alpha}}
\left(\frac{\alpha}{B}\right)^{\frac{2(2-\alpha)}{\alpha}}
t^{\frac{4-3\alpha}{\alpha}},
\end{equation}
by taking
\begin{equation}\label{eq:cmax}
c:=\frac{\alpha}{4-\alpha}\frac{A}{B}.
\end{equation}
\end{prop}
\begin{proof} The first statement \eqref{eq:c} follows immediately from \eqref{eq:energydissipation:1}. For \eqref{eq:energydissipation:2}, we need to maximize the functional
$$F(c)=\frac{2\pi}{\alpha}c^{\frac{2(2-\alpha)}{\alpha}}
\left(A-\frac{2(2-\alpha)}{\alpha}Bc\right).$$	
Since $F$ is concave with
$$
F'(c)
=\frac{2\pi}{\alpha}
\frac{2(2-\alpha)}{\alpha}
c^{\frac{4-3\alpha}{\alpha}}
\left(A-\frac{4-\alpha}{\alpha}Bc\right),
$$
it follows that $F$ attains its maximum at \eqref{eq:cmax} with
$$F(c)
=2\pi
\left(\frac{A}{4-\alpha}\right)^{\frac{4-\alpha}{\alpha}}
\left(\frac{\alpha}{B}\right)^{\frac{2(2-\alpha)}{\alpha}}.$$
This concludes the proof.
\end{proof}

\subsection{The profile $H$}

In this section we construct profiles $H$ satisfying the requirements from the previous sections.
We define $\mathcal{H}$ as the space of profiles $H\in  C^1([0,1])$ satisfying the conditions $H(0)=0$, $H(1)=1$, \eqref{eq:H:condition:1} and \eqref{eq:H:condition:2}. 

\begin{prop}\label{prop:H}
The space $\mathcal{H}$ is non-empty and convex.
\end{prop}
\begin{proof} First of all, notice that the four conditions $H(0)=0$, $H(1)=1$, \eqref{eq:H:condition:1} and \eqref{eq:H:condition:2} are convex.
Given $a,b>0$, we consider the ansatz
\begin{equation}\label{ansatz:H}
H(\xi):=(1-a\log\xi)\xi^{1+b}.
\end{equation}
It is clear that $H\in C^1([0,1])$ with $H(0)=0$ and $H(1)=1$.
On the one hand, an integration by parts yields
$$
\int_0^1\xi^2H\dif\xi
=\int_0^1(1-a\log\xi)\xi^{3+b}\dif\xi
=\frac{1}{4+b}\left(1+\frac{a}{4+b}\right).$$
Hence, the condition \eqref{eq:H:condition:1} is equivalent to
\begin{equation}\label{eq:a}
a:=\frac{(4+b)(\alpha+b)}{4-\alpha}.
\end{equation}
On the other hand, an integration by parts yields
\begin{equation}\label{eq:f}
\begin{split}
\int_0^1\xi H^2\dif\xi
=\int_0^1(1-a\log\xi)^2\xi^{3+2b}\dif\xi
&=\frac{1}{4+2b}\left(1+\frac{2a}{4+2b}\left(1+\frac{a}{4+2b}\right)\right)\\
&=\frac{1}{4(2+b)}\left(1+\left(1+\frac{a}{2+b}\right)^2\right)=:f(b).
\end{split}
\end{equation}
Notice that $f$ is continuous on $[0,\infty)$ and the condition \eqref{eq:H:condition:2} 
$$2(2-\alpha)f(b)<1,$$
is open. Then, it is enough to check that it is satisfied at $b=0$. Since 
\begin{equation}\label{eq:f0}
f(0)
=\frac{4^2+\alpha^2}{4(4-\alpha)^2},
\end{equation}
the condition \eqref{eq:H:condition:2} with $b=0$ is equivalent to $\alpha>0$.
\end{proof}

We finish this section by computing the energy dissipation rate for the particular ansatz $H$ given in \eqref{ansatz:H}. We consider the case $b=0$ to simplify the computations. For small $b$'s the result will be similar by continuity. We remark that, although the profile \eqref{ansatz:H} with $b=0$ is not differentiable at $\xi=0$, it still satisfies $H(0)=0$, which makes $\bar{v}$ continuous at $x=0$ for $t>0$. Moreover, it improves the regularity of the power-law vortex. For $b>0$ we have $H'(0)=0$, which makes $\bar{v}$ differentiable at $x=0$ for $t>0$. Similarly, it should be possible to construct profiles $H$ with better regularity.

\begin{prop}\label{prop:Eansatz}
Let $H$ be the profile \eqref{ansatz:H} with $b=0$. Then, the energy dissipation rate \eqref{eq:energydissipation:2} equals
$$\int_{\R^2}\partial_t\bar{e}\dif x
=-\frac{\pi}{16}
\left(\frac{2\alpha}{4-\alpha}\right)^{\frac{8-\alpha}{\alpha}}t^{\frac{4-3\alpha}{\alpha}},$$
and the growth rate \eqref{eq:cmax} equals
\begin{equation}\label{ansatz:c}
c=\left(\frac{2\alpha}{4-\alpha}\right)^2.
\end{equation}
\end{prop}
\begin{proof}
On the one hand, \eqref{def:A} equals
$$A=\frac{1}{2}-(2-\alpha)f(b),$$
where $f$ is given in \eqref{eq:f}. On the other hand, 
\begin{align*}
\xi^2 H-\frac{4-\alpha}{\xi}\int_{0}^{\xi}\zeta^2H\dif\zeta
&=\xi^{3+b}(1- a\log\xi)
-\frac{4-\alpha}{4+b}\xi^{3+b}\left((1- a\log\xi)+\frac{a}{4+b}\right)\\
&=\frac{(b+\alpha)^2}{4-\alpha}\xi^{3+b}|\log\xi|,
\end{align*}
where we have applied \eqref{eq:a}. Hence, \eqref{def:B} equals
$$B=\frac{(b+\alpha)^2}{4-\alpha}\int_0^1\xi^{3+b}|\log\xi|\dif\xi
=\frac{(b+\alpha)^2}{(4-\alpha)(4+b)^2}.$$
For $b=0$, these formulas simplify to (recall \eqref{eq:f0})
$$A=\frac{\alpha^3}{4(4-\alpha)^2},
\quad\quad
B=\frac{\alpha^2}{16(4-\alpha)}.$$
Therefore,
$$
\left(\frac{A}{4-\alpha}\right)^{\frac{4-\alpha}{\alpha}}
\left(\frac{\alpha}{B}\right)^{\frac{2(2-\alpha)}{\alpha}}
=\frac{1}{2^5}\left(\frac{2\alpha}{4-\alpha}\right)^{\frac{8-\alpha}{\alpha}}.
$$
This concludes the proof.
\end{proof}

\begin{figure}[h!]
	\centering
	\includegraphics[width=0.6\textwidth]{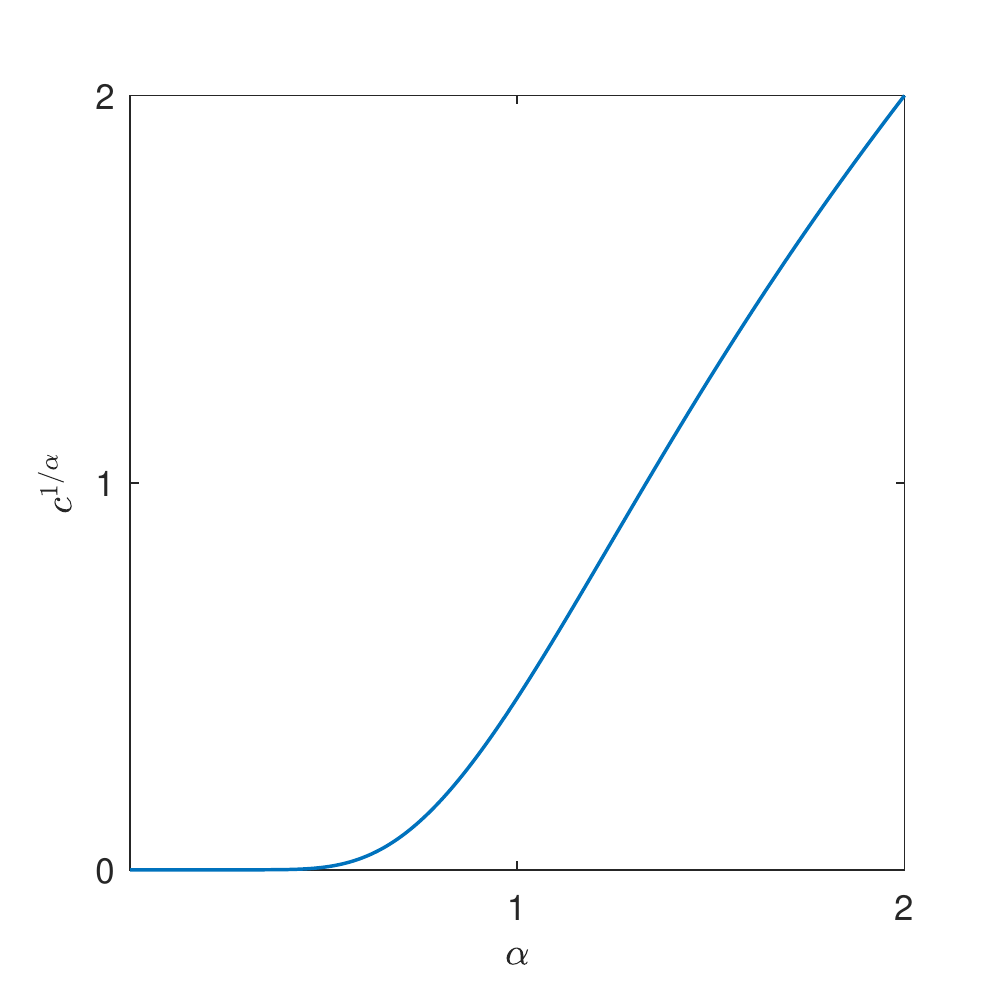}
	\caption{Plot of the growth rate $c(\alpha)^{\nicefrac{1}{\alpha}}$.}
	\label{fig:growth}
\end{figure}

\subsection{The truncation}\label{sec:truncation}
In this section we prove Theorems \ref{thm:nonuniqueness:Holder}, \ref{thm:nonuniqueness:Lp} \& \ref{thm:energy}. Let us fix $0<\alpha <2$.
Recall that the inconvenient of considering the self-similar subsolutions $(\bar{v},\bar{\sigma},\bar{q})$ from the previous sections is that they have infinite energy. This is because their tails are not integrable.
In order to make the energy finite, we fix $r_0>0$ and consider the truncated profile
$$h_\chi:=h\chi,$$
where $\chi:[0,\infty)\to[0,1]$ is a smooth cutoff with $\chi(r)=1$ if $r\in[0,r_0]$. 
On the one hand, the profile $h$ is determined by $H$ via \eqref{ansatz:h}, where we take $H$ as in Proposition \ref{prop:H} with $b=0$. This $H$ determines also the growth rate $c$ by \eqref{ansatz:c}, the terms $W_2,Q$ by Proposition \ref{prop:QW2}, and thus $(\bar{v},\bar{\sigma},\bar{q})$ by \eqref{ansatz:v}-\eqref{ansatz:q}, \eqref{ansatz:w1} and \eqref{ansatz:h}. On the other hand, $h_\chi$ determines the terms
$w_{\chi}$, $q_\chi$ and $\bar{e}_\chi$ by Corollary \ref{cor:radial}. The final subsolution $(\bar{v},\bar{\sigma},\bar{q})_{\chi}$ is defined by \eqref{ansatz:v}-\eqref{ansatz:q}, which agrees with $(\bar{v},\bar{\sigma},\bar{q})$ for $|x|\leq r_0$ provided that  $(ct)^{\nicefrac{1}{\alpha}}\leq r_0$.
Hence, for all $0\leq t\leq T$, where
\begin{equation}\label{eq:T}
T:=\frac{r_0^\alpha}{c}
=\left(\frac{4-\alpha}{2\alpha}\right)^2r_0^\alpha,
\end{equation} 
the energy dissipation rate of the (truncated) subsolution equals (recall Proposition \ref{prop:Eansatz})
$$\partial_t\bar{E}_{\chi}
=\int_{\R^2}\partial_t\bar{e}_\chi
=-\frac{\pi}{16}
\left(\frac{2\alpha}{4-\alpha}\right)^{\frac{8-\alpha}{\alpha}}t^{\frac{4-3\alpha}{\alpha}}.$$
Therefore, 
$$\bar{E}_{\chi}(t)
=E_{\chi}(0)
-\frac{\pi}{32}\frac{\alpha}{(2-\alpha)}
\left(\frac{2\alpha}{4-\alpha}\right)^{\frac{8-\alpha}{\alpha}}t^{\frac{2(2-\alpha)}{\alpha}},
$$
where
$$E_{\chi}(0)
=\frac{1}{2}\int_{\R^2}|v_s\chi|^2\dif x
\geq
\pi\int_0^{r_0}r^{3-2\alpha}\dif r
=\frac{\pi}{2(2-\alpha)}r_0^{2(2-\alpha)}.$$
We have proved Theorem \ref{thm:energy}. For Theorems \ref{thm:nonuniqueness:Holder} \& \ref{thm:nonuniqueness:Lp}, we invoke the h-principle for the Euler equation (Theorem \ref{thm:HP}) by taking some energy profile $e$ satisfying $\bar{e}<e$ on $\Omega=\{|x|<(ct)^{\nicefrac{1}{\alpha}}\}$. It is possible to select $e$ making the energy $E=\int e\dif x$ either constant or decreasing. Finally, notice that these solutions are uniformly bounded on $[0,T]\times\R^2$ if and only if $0<\alpha\leq 1$ due to \eqref{ansatz:h}.

\section{The case $\alpha\to 2$}\label{sec:alpha2}
In this section we analyze the borderline case $\alpha\to 2$. This corresponds to the point vortex $\omega_s=2\pi\delta_0$ for $\chi=1$. Notice that the initial velocity has infinite energy at the origin because $h^\circ(r)=\nicefrac{1}{r}$. However, our subsolution $(\bar{v},\bar{\sigma},\bar{q})$
has energy dissipation rate
$$\int_{\R^2}\partial_t\bar{e}
=-\frac{\pi}{2t},$$
and therefore it has finite energy for $t>0$
$$\bar{E}(t)=\bar{E}(1)-\frac{\pi}{2}\log t.$$
Hence, $(\bar{v},\bar{\sigma},\bar{q})$ is well defined in the Banach space $C_{\log t}(L^2\times L^1\times L^1)$, which is given by the (weighted) norm
$$\|f\|_{C_{\log t}L^p}
:=\sup_{t\in (0,T]}\frac{\|f(t)\|_{L^p}}{\max\{1,|\log t|^{\nicefrac{1}{p}}\}}.$$ 
This integrability class is enough to make sense of definition \eqref{eq:weak}. Moreover, the initial datum is attained in $L^{2-}$.
The h-principle for the Euler equation \cite{DeLellisSzekelyhidi09} can be easily modified to construct velocities in this class.
As a result, we show non-uniqueness of dissipative solutions to the Euler equation for the (truncated) point vortex datum.

\begin{thm} Let $0<\beta<1$. Then, there exist infinitely many weak solutions $v\in C_tL^{2-}$ to the Euler equation starting from
$$v_s(x)=\chi(|x|)\frac{x^\perp}{|x|^2}.$$
Furthermore, $v=v_s$ outside $\{|x|\leq 2\sqrt{t}\}$, and
$v\in C_{\log t}L^2$ with
$$\partial_t E=-\beta\frac{\pi}{2t},$$ 
for all $t\in (0,T]$, where $T$ is given in \eqref{eq:T}.
\end{thm}

\section{Acknowledgements}
The author thanks \'Angel Castro, Daniel Faraco, Francisco Gancedo, Antonio Hidalgo and L\'aszl\'o Sz\'ekelyhidi for stimulating discussions during the preparation of this work.
This research started in Princeton during the Special Year 2021-22: h-Principle and Flexibility in Geometry and PDEs. 
This work owes a great deal to the Analysis Seminar, as well as financial support provided by the Institute for Advanced Study. 
Part of the work took place at the University of Sevilla. The author would like to thank its financial support and its friendly atmosphere.
This work was finished in Leipzig. The author would like to thank the excellent working conditions and financial support provided by the Max Planck Institute for Mathematics in the Science.
The author acknowledge financial support from the Spanish Ministry of Science and Innovation through the Severo Ochoa Programme for Centres of Excellence in R\&D
(CEX2019-000904-S), the grants PID2020-114703GB-I00 and MTM2017-85934-C3-2-P, and the ERC Advanced Grant 834728.

\bibliographystyle{abbrv}
\bibliography{Nonuniqueness_Lp}

\begin{flushleft}
	\quad\\
	\textsc{Max Planck Institute for Mathematics in the Sciences\\
	04103 Leipzig, Germany}\\
	\textit{E-mail address:} \href{mailto:fmengual@mis.mpg.de}{\nolinkurl{fmengual@mis.mpg.de}}
\end{flushleft}

\end{document}